\documentclass{amsart}
\usepackage{mathrsfs}
\usepackage{amsfonts}
\usepackage{txfonts}
\usepackage{amsmath}
\usepackage{amssymb}
\usepackage{amsthm}
\usepackage{graphicx}
\usepackage[toc,page,title,titletoc,header]{appendix}
\usepackage{geometry}
\usepackage{upgreek}
\usepackage[T1]{fontenc}
\usepackage{color}
\usepackage{hyperref}
\usepackage[all]{xy}
\usepackage[french,english]{babel}

\theoremstyle{plain}
\newtheorem{theo}{Theorem}[section]
\newtheorem*{theo*}{Theorem}
\newtheorem{prop}{Proposition}[section]

\newtheorem*{lem*}{Lemma}

\theoremstyle{definition}

\newtheorem{defi}{Definition}[section]

\theoremstyle{remark}

\newtheorem*{rem*}{Remark}

\newcommand{\R}{\mathbb{R}}

\newcommand{\C}{\mathbb{C}}

\begin{document}

\title{Projective Dynamics and an Integrable Boltzmann Billiard Model}

\author{Lei Zhao}

\abstract 
The aim of this note is to explain the integrability of an integrable Boltzmann billiard model, previously established by Gallavotti and Jauslin \cite{Gallavotti-Jauslin}, alternatively via the viewpoint of projective dynamics. The additional first integral is shown to be related to the energy of a corresponding system on a hemisphere. We show that this viewpoint applies to certain other billiard models in the plane defined through Kepler-Coulomb problems as well. The approach also leads to a family of integrable billiard models on the sphere defined through the spherical Kepler-Coulomb problem.

\endabstract

\date\today
\maketitle

\section{Introduction}
In \cite{Gallavotti-Jauslin}, Gallavotti and Jauslin examined a billiard model derived from the Kepler problem in the plane, with a line not containing the attractive center as wall of reflection. The billiard model is defined through the Kepler dynamics with the usual law of reflection in the plane. In the case of the Kepler problem, since each Keplerian orbit with negative energy intersect the line of reflection generically at two points, the billiard dynamics on one side of the wall is the reverse of the billiard dynamics on the other side. Still it is preferred to take the dynamics on the other side of the line as the center to avoid problem of collisions. Gallavotti and Jauslin showed that the billiard system is integrable in the sense that it has another first integral additional to the energy, which confirms a previous conjecture of Gallavotti \cite{Gallavotti 2014}, \cite[Appendix D]{Gallavotti} on its integrability.  This integrable Boltzmann model is shown to carry periodic and quasi-periodic dynamics by Felder \cite{Felder} with algebraic geometric method related to the Poncelet theorem.

This model is a limiting case of a toy model considered by Boltzmann \cite{Boltzmann}, defined such that an additional centrifugal force with strength inverse proportional to the cubic of the distance to the center is added. Boltzmann assumed that this is an ergodic system which illustrates his ``ergodic hypothesis''.  Actually this may only happen when the additional centrifugal force is sufficiently large, as by the analysis of Gallavotti-Jauslin and Felder, KAM and topological stability of orbits on energy hypersurfaces can be established via application of KAM theory.  

In this note we aim to explain the integrability of the integrable Boltzmann Model alternatively with the viewpoint of projective dynamics, as developed and illustrated in \cite{Halphen}, \cite{Appell}, \cite{Albouy 2008}, \cite{Albouy 2013}. We shall show that the first integrals of Gallavotti-Jauslin is equivalent to the energies of the planar problem and a corresponding spherical problem. Moreover, we explain that this viewpoint leads to certain integrable variants of this integrable Boltzmann model in the plane and on the sphere.

We note that the ``projective method'' has been previously used to obtain the integrability of the geodesic flow on an ellipsoid \cite{Tabachnikov 1999}, \cite{Topalov-Matveev} and the billiard system inside an ellipsoid \cite{Tabachnikov 2002}. Our result can be considered as extension of this method to billiard systems defined through mechanical systems. 

We organize this note as follows: We first recall projective dynamical properties of the Kepler-Coulomb problem in Section \ref{Sec: 2}. In Sections \ref{Sec: Int Models plan} and \ref{Sec: 4}, we define certain billiard systems in the plane (which include the integrable Boltzmann model) and on the sphere with Kepler-Coulomb problem and prove their integrability respectively. Some further remarks are collected in Section \ref{Sec: 5}.


\section{Projective Dynamics of Kepler-Coulomb Problems}\label{Sec: 2}
We start with a mechanical system $(M, g, U)$ on a Riemannian manifold $(M, g)$ with the force function (negative of the potential) $U$. Newton's equations of motion of this system are given by 
$$\nabla_{\dot{q}} \dot{q}=\hbox{grad } U,$$
in which $\nabla$ denotes the Levi-Civita connection associated to $g$, and $\hbox{grad}$ denotes the gradient. 

Following Painlev\'e \cite{P}, we define

\begin{defi} Two mechanical systems $(M_{1}, g_{1}, U_{1}), (M_{2}, g_{2}, U_{2})$ are called \emph{in correspondence}, if there is a diffeomorphism $\phi: M_{1} \to M_{2}$, such that for any unparametrized orbit $\gamma \subset M_{1}$ of  $(M_{1}, g_{1}, U_{1})$, the orbit $\phi(\gamma) \subset M_{2}$ is an unparametrized orbit of $(M_{2}, g_{2}, U_{2})$ and vice versa.
\end{defi}

In other words, two mechanical systems are in correspondence if their orbits in configuration spaces agree up to diffeomorphism and up to time-reparametrization. With this definition, a mechanical system $(M, g, U)$ naturally have corresponding systems $(M, k g, k' U), k,k' \in \R_{+}$, but these do not provide further informations of the system, thus this type of correspondence is considered as trivial. The correspondence is called non-trivial if the energies of the two mechanical systems are functionally independent. When a mechanical system has a non-trivial corresponding system, then the system itself can be realized as a quasi-bi-Hamiltonian system, \emph{i.e.} it admits two different Hamiltonian formalisms up to time reparametrization and in this case both Hamiltonians, \emph{i.e.} both energies, are first integrals of the system \cite{BCRR} . Note that non-trivial corresponding system to a mechanical system, when exist, need not to be unique up to trivial correspondences.

Now let $\phi: M_{1} \to M_{2}$ be a diffeomorphism and suppose that we have a mechanical system $(M_{1}, g_{1}, U_{1})$.  With these we get a vector field (force field) $\phi_{*} \hbox{grad } U_{1}$ on $M_{2}$. To get a corresponding system of $(M_{1}, g_{1}, U_{1})$ it is enough to have a function $\rho: M_{2} \to \R_{+}$ such that the reparametrized force field $\rho \cdot \phi_{*} \hbox{grad } U_{1}$ is the gradient of a function $U_{2}: M_{2} \to \R$ with respect to a metric $g_{2}$. 

Now suppose $(M_{1}, g_{1}, V_{1})$ and $(M_{2}, g_{2}, V_{2})$ are in correspondence. Any initial condition $(x,v)\in T M_1$ determines an orbit $a(t)$ of $(M_{1}, g_{1}, U_{1})$ such that
$(a(0),\dot{a}(0))=(x,v)$. Then $\{b(t):=\phi \circ a(t)\}$ is an unparametrized orbit of the second system. This means that there is (at least locally) a diffeomorphism $\tau \mapsto t(\tau)$ such that $b(t(\tau))$ is an
orbit of $(M_{2}, g_{2}, U_{2})$. Due to the time reparametrization, to get the corresponding velocity in $M_{2}$, one needs to further multiply the push-forward $\phi_{*} v$ of the velocity by the factor 
$f:=d t/d \tau=\frac{d b(t(\tau))}{d \tau}/\frac{d b(t)}{d t}|_{t=0}$ which can be considered as a real function on $T M_{1}$.
The kinetic energy of the system $(M_{2}, g_{2}, U_{2})$ is thus expressed as a function on $T M_1$ as
$\dfrac{1}{2} g_2 (\phi(x)) (f\phi_* v,f\phi_*v)$, and the force
function as a function on $M_1$ as $U_2(\phi(x))$. In
combination this expresses the energy of the system $(M_{2}, g_{2}, U_{2})$ as a function on $T
M_1$.  

A planar central force problem is a mechanical system $(\R^{2} \setminus Z, g_{flat}, U)$ such that the force function $U$ is invariant under the $SO(2)$ action by rotations in $\R^{2}$ fixing the center $Z \in \R^{2}$. The mass factor of such a system refers to that of the center, which can nevertheless take both signs to allow both attractive and repulsive forces. The mass of the moving particle is taken as the unit of mass. A central force problem on a sphere is defined analogously.

We now consider two types of diffeomorphisms given by central projections and see how they transform central force problems, following \cite{Albouy 2008} and \cite{Albouy}. The center of projection is always $O=(0,0,0) \in \R^{3}$. 

The first type of projection projects affine planes to affine planes in $\R^{3}$ (See Figure \ref{Figure: ProjPlan}). Let $V_{1}, V_{2} \subset \R^{3}$ be two hyperplanes in $\R^{3}$ given respectively by the equations $\langle h_{1}, q\rangle=1$ and $\langle h_{2}, q\rangle=1$, in which $h_{1}, h_{2} \in {\R^{3}}^{*} \setminus O \cong \R^{3} \setminus O$. For a central force system on $V_{1}$, its center $Z_{1}$ is projected to a point $Z_{2}  \in V_{2}$. A point $q_{2} \in V_{2}$ is projected from a point $q_{1} \in V_{1}$ such that $q_{2}:= \lambda(q_{1})^{-1} q_{1}$. With the equations we determine the function as $\lambda(q_{1})=\langle h_{2} , q_{1}\rangle$. We consider the projection in the region where $\lambda(q_{1})=\langle h_{2} , q_{1}\rangle =\langle h_{1} , q_{2} \rangle^{-1}>0$, which projects the half plane $\{q_{1} \in V_{1}, \langle h_{2} , q_{1}\rangle>0\}$ to the half plane $\{q_{2} \in V_{2}, \langle h_{1} , q_{2}\rangle>0\}$.

\begin{figure}\label{Figure: ProjPlan}
\center
\includegraphics[width=80mm]{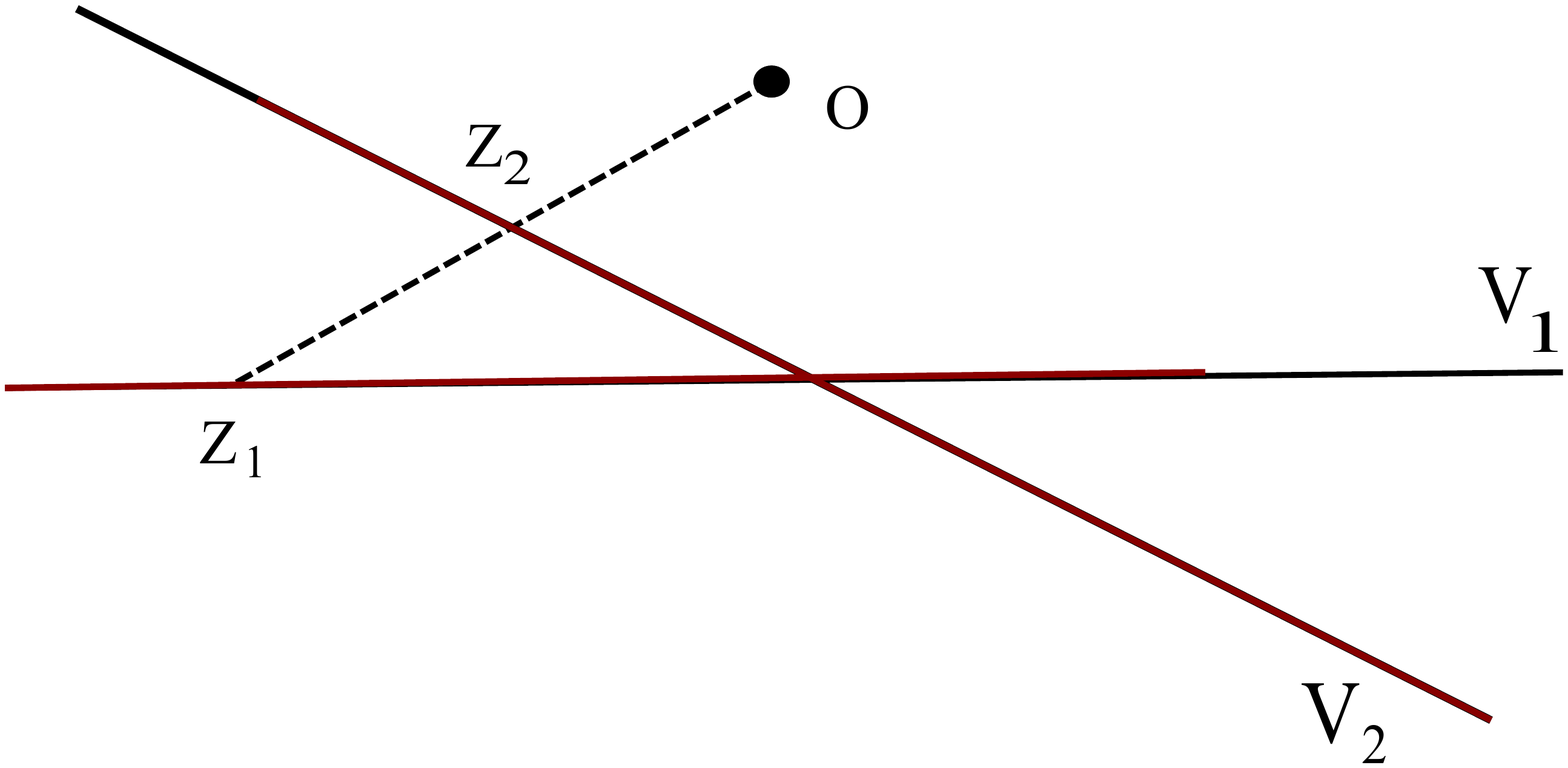}
\caption{A Plane-Plane Projection}
\end{figure} 

We compute
$$\dot{q}_{2}=\lambda(q_{1})^{-2} (\dot{q}_{1} \lambda(q_{1})-\langle \hbox{grad } \lambda(q_{1}), \dot{q}_{1} \rangle q_{1})=\lambda(q_{1})^{-2} (\dot{q}_{1} \lambda(q_{1})-\langle h_{2}, \dot{q}_{1} \rangle q_{1}).$$
We now change time in $V_{2}$ according to the law $\dfrac{d}{d \tau}=\lambda(q_{1})^{2} \dfrac{d}{d t}$. We denote the time derivative with respect to $\tau$ by $'$. We may then write the above expression as
$$q'_{2}=\dot{q}_{1} \lambda(q_{1})-\langle h_{2}, \dot{q}_{1} \rangle q_{1},$$
and 
\begin{equation}\label{eq: q2''}
q^{''}_{2}=\lambda(q_{1})^{2} (\lambda(q_{1}) \ddot{q}_{1}-\langle h_{2}, \ddot{q}_{1} \rangle q_{1}).
\end{equation}

We also have $q_{1}=\langle h_{1},  q_{2}\rangle^{-1} q_{2}$. The right hand side can thus be written in a form depending only on $q_{2}$ and therefore defines a force field on $V_{2}$. 

We now assume that $U_{1}=m_{1} \|q_{1}-Z_{1}\|_{1}^{\alpha}$, in which the distance $\|\cdot\|_{1}$ is an Euclidean distance on $V_{1}$. It can be taken as the restriction of the Euclidean distance of $\R^{3}$, but it has to be kept in mind that this is purely auxiliary and we shall not always equip $\R^{3}$ with its Euclidean form.

We thus have 
$$\ddot{q}_{1}=\hbox{grad } U_{1}=\alpha m_{1} \|q_{1}-Z_{1}\|_{1}^{\alpha-2} (q_{1}-Z_{1}).$$
Thus Eq. \eqref{eq: q2''} takes the form
\begin{equation}\label{eq: q2''central}
\begin{aligned}
q^{''}_{2}=&\lambda(q_{1})^{2} (\lambda(q_{1}) \ddot{q}_{1}-\langle h_{2}, \ddot{q}_{1} \rangle q_{1}) \\
            = &\alpha m \langle h_{1}, q_{2} \rangle^{-3}  \langle h_{1}, Z_{2} \rangle^{-1} \|  q_{2} \langle h_{1}, q_{2} \rangle^{-1}-Z_{2} \langle h_{1}, Z_{2} \rangle^{-1}\|_{1}^{\alpha-2} (q_{2}-Z_{2}). 
\end{aligned}
\end{equation}

To proceed we would like to sort out the factor $\langle h_{1}, q_{2} \rangle^{-1}$ in the distance appeared in the above expression. On the other hand, the vector $Z_{2}  \langle h_{1}, q_{2} \rangle  \langle h_{1}, Z_{2} \rangle^{-1} \in \R^{3}$ does not necessarily lie in $V_{2}$. 

To define a system in $V_{2}$, it is desired to define a distance on $V_{2}$ such that  
\begin{equation}\label{eq: desired property}
\|  q_{2}-Z_{2} \|_{2}=\langle h_{1}, q_{2} \rangle \|  q_{2} \langle h_{1}, q_{2} \rangle^{-1}-Z_{2} \langle h_{1}, Z_{2} \rangle^{-1}\|_{1}.
\end{equation}
In order to achieve this, we extend the Eucldean distance $\|\cdot\|_{1}$ in a non-standard way: Each vector $v \in \R^{3}$ is decomposed as $v=v_{1} + c \cdot Z_{1}$. We set $\|v\|_{*}=\|v_{1}\|_{1}$. Consequently we set $\|\cdot\|_{2}$ as the restriction of $\|\cdot\|_{*}$ to $V_{2}$. Clearly $\|\cdot\|_{2}$, which is non-degenerate whenever $V_{1}$ and $V_{2}$ are not perpendicular, has the desired property \eqref{eq: desired property}.

We thus deduce from Eq. \eqref{eq: q2''central}, by setting $m_{2}=m_{1} \langle h_{1}, Z_{2} \rangle^{-1}$, that
\begin{equation}\label{eq: q2''central}
\begin{aligned}
q^{''}_{2}=\alpha m_{1} \langle h_{1}, Z_{2} \rangle^{-1} \langle h_{1}, q_{2} \rangle^{-3-(\alpha-2)} \|q_{2}-Z_{2}\|^{\alpha-2} (q_{2}-Z_{2})=  \langle h_{1}, q_{2} \rangle^{-1-\alpha} \hbox{grad } m_{2}   \|q_{2}-Z_{2}\|_{2}^{\alpha},
\end{aligned}
\end{equation}
which is again the force field of a central force problem with potential when $\alpha=-1$, \emph{i.e.} when the system is the Coulomb-Kepler problem. 

We summarize
\begin{prop}\label{prop: plane to plane} A Kepler-Coulomb problem with mass $m_{1}$ and center $Z_{1}$ in a half-plane in $V_{1}$ is centrally projected to a Kepler-Coulomb problem with mass $m_{2}$ and center $Z_{2}$ in a half-plane in $V_{2}$.
\end{prop}

We now discuss central projection from a hemisphere to a plane. We take $S$ to be the unit sphere in $\R^{3}(x, y, z)$. We consider the plane $V:=\{z=-1\}$ (so $h=(0,0,-1)$) tangent to $S$ at its south pole $(0, 0, -1)$ which we take as the center $Z$. The central projection from the origin $O \in \R^{3}$ defines a diffeomorphism between the south-hemisphere $S_{SH}:=S \cap \{z <0\}$ and $V$. It is well-known that this projection sends unparametrized geodesics to unparametrized geodesics.

We now let $q_{1} \in V$ and $q_{2} \in S_{SH}$ be related by the central projection as: 
\begin{equation}\label{eq: time change}
q_{2}= (1+ \|q_{1}-Z\|^{2})^{-1/2} q_{1}:= \lambda(q_{1})^{-1} q_{1}.
\end{equation}
Again we have
$$\dot{q}_{2}=\lambda(q_{1})^{-2} (\dot{q}_{1} \lambda(q_{1})-\langle \hbox{grad } \lambda(q_{1}), \dot{q}_{1} \rangle q_{1}),$$
and we may again change time according to $\dfrac{d}{d \tau}=\lambda(q_{1})^{2} \dfrac{d}{d t}$ which then leads to the equation
\begin{equation}\label{eq: q2'' sph}
q^{''}_{2}=\lambda(q_{1})^{2} (\lambda(q_{1}) \ddot{q}_{1}-\langle \hbox{grad } \lambda(q_{1}), \ddot{q}_{1} \rangle q_{1}).
\end{equation}
The last term of the right hand side of this equation is proportional to $q_{2}$, thus is vertical to $S_{SH}$. It can be seen as a force of constraint which keeps the force to be tangent to $S_{SH}$. By replacing $\ddot{q}_{1}$ with a force field, we see that with this computation, any force field in $V$ is transformed into a force field of $S_{SH}$ and vice versa. 

Now start with the Kepler-Coulomb problem in $V$, \emph{i.e.} 
$$\ddot{q}_{1}=-m \|q_{1}-Z\|^{-3} (q_{1}-Z),$$
we have 
$$\lambda(q_{1})^{3} \ddot{q}_{1}=-m \sin^{-3} \theta \cdot (q_{1}-Z),$$
in which $\theta$ denotes the angle $\angle ZOq_{1}$, whose component tangent to $S_{SH}$ at $q_{2}$ gives the force, which has norm $|m| \sin^{-2} \theta$, points toward $Z$ when $m>0$, and points in the reverse direction when $m<0$. This is a central force system on $S_{SH}$ with force function $m \cot \theta$.  We conclude that the resulting system is the spherical Kepler-Coulomb problem with mass factor $m$, a system first defined by Serret \cite{Serret}, on the hemisphere $S_{SH}$.  

In the spherical Kepler-Colomb problem, the force function $m \cot \theta$ (and indeed the whole system) extends to the whole $S$ except at the two poles, which are singular centers for the extended systems, one attractive and the other one repulsive. The sign of the mass factor $m$ now determines which center is attractive and which is repulsive. The orbits of this system are (possibly degenerate) spherical (conics and more precisely) ellipses. Elliptic, parabolic and hyberbolic orbits of the Kepler-Coulomb problem in $V$ correspond respectively to intersections with $S_{SH}$ of spherical ellipses lying entirely in $S_{SH}$, tangent to the equator and intersecting transversely the equator respectively. In this sense, the spherical Kepler-Coulomb problem completes and extends the corresponding planar Kepler-Coulomb problem.  We refer to \cite{Albouy} for more detailed analysis on the spherical Kepler-Coulomb problem.

We summarize the above analysis in the following proposition:
\begin{prop}\label{prop: plane to sphere} A Kepler-Coulomb problem in $V$ with mass factor $m$ and center at the south pole $(0,0,-1)$ is centrally projected to a spherical Kepler-Coulomb problem in $S_{SH}$ with mass factor $m$ and with center at $(0,0,-1)$. 
\end{prop}


We now construct a projection in which the south pole is not necessarily the center of the spherical Kepler-Coulomb problem.
We let $Z=(0, a, -1) \in V$ be the center in $V$, which is the projection of the center $Z_1:=(0, \dfrac{a}{\sqrt{1+a^{2}}}, -\dfrac{1}{\sqrt{1+a^{2}}})$ in $S_{SH}$. With the previous procedure, the spherical Kepler-Coulomb force field is projected to an analytic force field $F$ in $V$. To deduce $F$ without additional computation, we proceed with steps. We set $h_{1}=(0, \dfrac{a}{\sqrt{1+a^{2}}}, -\dfrac{1}{\sqrt{1+a^{2}}})$. The affine plane $V_{1}=\{q: \langle h_{1}, q \rangle =1\}$ is now tangent to $S_{SH}$ at $Z_1$. We first project from $S_{SH}$ to $V_{1}$, then from $V_{1}$ to $V$ (See Figure \ref{Figure: ProjSph}). Note that due to the composition of two different projections, a priori only an open subset of $S_{SH}$ is projected to an open subset of $V$. Nevertheless, the previous analysis shows that the force field $F$ coincides with that of the planar Kepler-Coulomb problem with mass factor $m$ on an open subset of $V$. Consequently by analyticity, $F$ is derived from the Kepler-Coulomb potential with mass factor $m$ in $V$. 

\begin{figure}\label{Figure: ProjSph}
\center
\includegraphics[width=80mm]{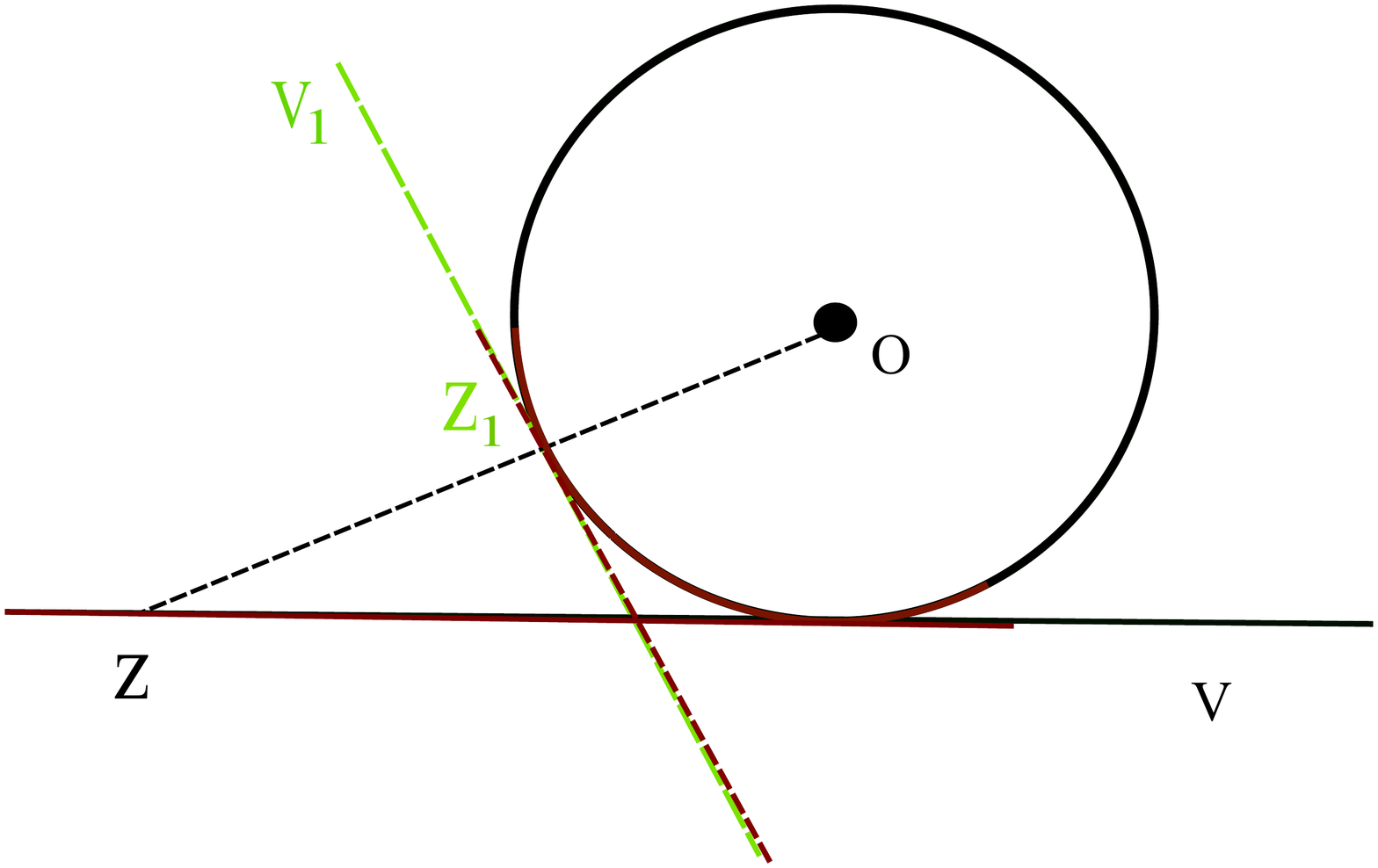}
\caption{A Plane-Sphere Projection}
\end{figure} 

We thus obtain the following proposition by combining Prop \ref{prop: plane to plane} and Prop \ref{prop: plane to sphere}.
\begin{prop}The planar Kepler-Coulomb problem in $V$ with mass factor $m$ and center $(0, a, -1)$ is centrally projected from the spherical Kepler-Coulomb problem on $S_{SH}$ with mass factor $m'=m \sqrt{1+a^{2}}$ and center $(0, \frac{a}{\sqrt{1+a^{2}}}, -\frac{1}{\sqrt{1+a^{2}}})$.
\end{prop}

As a consequence, the energies $E_{pl}, E_{sph}$ of the planar and the spherical problems are both first integrals of the planar Kepler-Coulomb problem. They are also functionally independent, as can be observed with their explicit expressions in a chart which we shall describe later.

Moreover, we shall see that $\{E_{pl}, E_{sph}\}$ is equivalent to $\{E_{pl}, D\}$ where $D$ is Gallavotti-Jauslin's first integral for the integrable Boltzmann billiard model, constructed with an analysis of the geometry of ellipses \cite{Gallavotti-Jauslin}. Our result shows that the same is true for parabolic and hyperbolic orbits of the planar Kepler-Coulomb problem as well.

We now go through necessary computations. Let $S$ be the unit sphere in $\R^{3}$ and $V_{1}$  be the affine plane tangent to $S$ at the point $Z_{1}=\Bigl(0, \dfrac{a}{\sqrt{1+a^{2}}}, -\dfrac{1}{\sqrt{1+a^{2}}}\Bigr)$ which is the center of the spherical Kepler-Coulomb problem with mass $m'$. We equip both $S$ and $V_{1}$ with the induced metrics from the standard Euclidean metric of $\R^{3}$. We take $V=\{z=-1\}$, which is tangent to $S$ at $(0,0,-1)$ on which there defines a Kepler-Coulomb problem with center $Z:=(0, a,-1)$ and with mass $m$. The metric $\|\cdot\|_{2}$ in $V$ will not be the one induced from the standard Euclidean metric of $\R^{3}$, but the one defined by Prop. \ref{prop: plane to plane}, which we shall now compute. 

According to the construction we first extend the Euclidean form $\|\cdot\|_{1}$ of $V_{1}$ to a non-standard form $\|v\|_{*}$ on $\R^{3}$, by setting for $v \in \R^{3}$, that $\|v\|_{*}=\|v-\langle v, Z_{1} \rangle Z_{1}\|_{1}$. Now for $v =(x, y, -1) \in V$, we have 
$$v-\langle v, Z_{1} \rangle Z_{1}=\Bigl(x, y-\dfrac{a (ay+1)}{1+a^{2}},-1+\dfrac{ay+1}{1+a^{2}}\Bigr)=\Bigl(x, \dfrac{y-a}{1+a^{2}},\dfrac{a(y-a)}{1+a^{2}}\Bigr),$$ 
thus
$$\|v\|_{2}=\|v\|_{*}=\sqrt{x^{2}+\dfrac{(y-a)^{2}}{1+a^{2}}}.$$

We therefore have that the energy of the planar Kepler-Coulomb problem in $(V, \|\cdot\|_{2})$ is
\begin{equation}\label{eq: E_{pl,2}}
E_{pl, 2}=\dfrac{1}{2} (\dot{x}^{2}+\dfrac{\dot{y}^{2}}{1+a^{2}})-\dfrac{m}{\sqrt{x^{2}+\dfrac{(y-a)^{2}}{1+a^{2}}}}.
\end{equation}

We now express the energy of the spherical Kepler-Coulomb problem in $S_{SH}$ with center $Z_{1}$ and with mass $m'$ in $V$, by viewing $V$ as a gnomonic chart for $S_{SH}$ as given by the central projection. 

In the gnomonic chart $V$, the metric of $S_{SH}$ takes the form 
$$\dfrac{1}{(1+x^{2}+y^{2})^{2}} ((1+y^{2}) d x^{2} -2 xy d x d y + (1+x^{2}) d y^{2}).$$ 
The kinetic energy of the system thus takes the form
\begin{equation}\label{eq: kinetic sph}
\dfrac{1}{2(1+x^{2}+y^{2})^{2}} ((1+y^{2}) {x'}^{2} -2 xy x' y' + (1+x^{2}) d {y'}^{2})
\end{equation}
in which $'$ denotes the derivative with respect to the time on $S_{SH}$. The factor of time change is obtained via \eqref{eq: time change}:  
$$(x', y')=(1+x^{2}+y^{2}) (\dot{x}, \dot{y}).$$
Thus the kinetic energy is expressed equivalently as
$$ \dfrac{1}{2}((1+y^{2}) \dot{x}^{2} -2 xy \dot{x} \dot{y} + (1+x^{2}) \dot{y}^{2})=\dfrac{1}{2} \Bigl((\dot{x}^{2}+\dot{y}^{2})+(x \dot{y}-y \dot{x})^{2}\Bigr).$$
We see that this is the sum of the kinetic energy in the plane and half of the square of the angular momentum in the plane with respect to the point $(x, y)=(0,0)$. 

The potential on the sphere is $-m' \cot \theta$, where $\theta$ is the angle between the position of the moving particle $\dfrac{1}{\sqrt{x^{2}+y^{2}+1}}(x, y, -1)$ on $S_{SH}$ and the center $Z_{1}$. This is expressed in terms of $(x, y)$ as
$$-m' \dfrac{a y+1}{ \sqrt{(y-a)^{2}+(1+a^{2}) x^{2}}},$$
thus the spherical energy has expression
$$E_{sph, 2}:=\dfrac{1}{2} ((1+y^{2}) \dot{x}^{2} -2 xy \dot{x} \dot{y} + (1+x^{2}) \dot{y}^{2}) -m' \dfrac{a y+1}{ \sqrt{(y-a)^{2}+(1+a^{2}) x^{2}}}.$$

We now normalize the metric in $V$ to a standard Euclidean one via the affine transformation 
$$(\xi, \eta) \mapsto (x=\xi, y=\sqrt{1+a^{2}} \eta+a),$$ which implies $(\dot{x}, \dot{y})=(\dot{\xi}, \sqrt{1+a^{2}} \dot{\eta}).$ The planar and spherical energies now respectively take the forms (with relation $m'= \sqrt{1+a^{2}}\, m$)
$$E_{pl}=\dfrac{1}{2} (\dot{\xi}^{2}+\dot{\eta}^{2})-\dfrac{m}{\sqrt{\xi^{2}+\eta^{2}}},$$
and
$$E_{sph}=(1+a^{2}) \Bigl(\dfrac{1}{2}  (\dot{\xi}^{2}+\dot{\eta}^{2}) -\dfrac{m}{\sqrt{\xi^{2}+\eta^{2}}} \Bigr) + \dfrac{(1+a^{2})}{2} \Bigl((\xi \dot{\eta} - \eta \dot{\xi})^{2}-\dfrac{2 a}{\sqrt{1+a^{2}}}((\xi \dot{\eta} - \eta \dot{\xi}) \dot{\xi}+\dfrac{m \eta}{\sqrt{\xi^{2}+\eta^{2}}})\Bigr).$$

Set $L=\xi \dot{\eta} - \eta \dot{\xi}$ the angular momentum with respect to $(\xi, \eta)=(0, 0)$ and $A_{\eta}=-L \dot{\xi}-\dfrac{m \eta}{\sqrt{\xi^{2}+\eta^{2}}}$ the $\eta$-component of the Laplace-Runge-Lenz vector. Let $h=-\dfrac{a}{\sqrt{1+a^{2}}}$ be the relative distance from the center $(0, a, -1)$ to the line $\mathcal{L}:=\{(x, 0, -1)\}$ in $V_{2}$ with respect to the metric $\| \cdot \|_{2}$, and $D=L^{2}-2 h A_{\eta}$. We have
\begin{equation}\label{Eq: Esp D}
E_{sph}=(1+a^{2}) (E_{pl} + \dfrac{D}{2}).
\end{equation}
Consequently for the planar system, we have that 

\begin{prop} The pair of independent first integrals $\{E_{pl}, D\}$ is equivalent to the pair of independent first integrals $\{E_{pl}, E_{sph}\}$. 
\end{prop}


\section{Integrable Billiard Models with Potentials in the Plane}\label{Sec: Int Models plan}

We now define certain billiard models in the plane with Kepler-Coulomb problem with mass factor $m$ and center $Z$. For this, it is enough to specify the wall of reflection and on which side of the wall the dynamics defining the billiard system takes place. At this generality, not all orbits which start from a point on the wall comes back to hit the wall again, so that the billiard mapping is possibly not always defined. Nevertheless we also accept such cases for uniformity. 

Such a billiard system is called \emph{integrable} when there exist two functionally independent conserved quantities. Note that the energy is always a conserved quantity. Therefore the system is integrable when an additional conserved quantity functionally independent from the energy can be found.

We consider the following billiard models with potentials (See Figure \ref{Figure: Five}): We take a Kepler-Coulomb problem in the plane, in which the mass factor can take either positive or negative signs. The wall of reflection is taken either as a circle $\mathcal{C}$ centered at the center of the Kepler-Coulomb problem (below we call such $\mathcal{C}$ centered circle), or a line $\mathcal{L}$, which can take any positions and the dynamics on either sides of the wall can be taken to define the billiard system. 

In order to deal with possible collisions with the center $Z$ in the line case, we have to invoke a regularization. Standard Levi-Civita regularization \cite{levi-civita2} does well the work, and gives a continuation of an orbit running into a collision with the center by an orbit ejecting from the collision by elastic bouncing. When the wall of reflection contains the center $Z$, any orbit which passes through a point in the line of reflection different from $Z$ will never meet $Z$, therefore it is also harmless to delete $Z$ from the line of reflection and consider only the rest, non-collisional orbits.

The integrable Boltzmann model corresponds to the case that the mass factor is negative, with a line as the wall of reflection, and (preferably) with the dynamics on the side of the line not containing the center $Z$.


\begin{figure}\label{Figure: Five}
\center
\includegraphics[width=80mm]{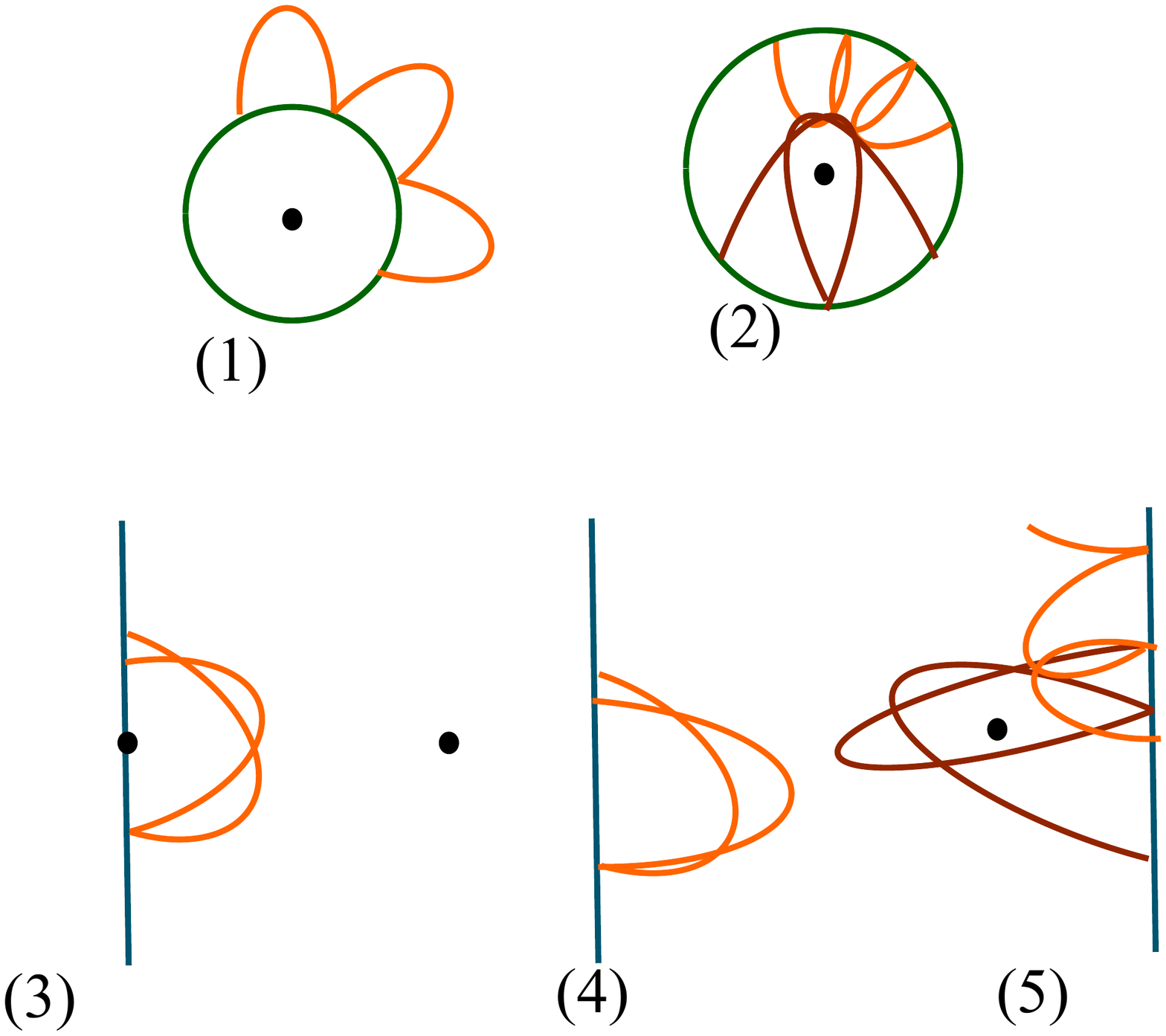}
\caption{Some planar billiard systems}
\end{figure} 


\begin{theo} Any of the planar billiard systems thus defined is integrable. The first integrals are its energy together with the energy of a spherical corresponding system.
\end{theo}

\begin{proof} The integrability of these systems are explained conveniently using projective dynamics. We let $V=\{z=-1\} \in \R^{3}$ be the plane in which the system is defined. 

We first consider the case that the wall is a centered circle $\mathcal{C}$. We choose it to have its center at the south pole $(0,0,-1)$ of $S$ which is thus also a center of the Kepler-Coulomb problem. Thus $a=0$ in \eqref{eq: E_{pl,2}} and the law of reflections, which preserves the planar kinetic energy is seen to also preserve the spherical kinetic energy \eqref{eq: kinetic sph}. In this case, it is also direct to see that the square of the angular momentum is preserved under reflections. 

When the wall is a line $\mathcal{L}$, we choose it to cross the south pole $(0,0,-1)$ of $S$ such that $(0,0, -1)$ is the nearest point to the center of the system $Z$ on $\mathcal{L}$. By translation and rotation we may set the line of reflection to $\mathcal{L}=\{(x,0,-1)\}$, and we may put the center at $(0, a,-1 )$. Note that in this case the plane is equipped with the metric $\|\cdot \|_{2}$ which is asymmetric, nevertheless it is seen that both the planar and spherical kinetic energies are preserved under the reflections, as follows from \eqref{eq: E_{pl,2}}, \eqref{eq: kinetic sph}.

Consequently both the planar and spherical energies are first integrals of the system in both cases.
\end{proof}

Note that in the case that the wall of reflection is a line, we may equivalently take the planar energy together with Gallavotti-Jauslin's first integral $D$ as a pair of first integrals.


\section{Integrable Billiard Models on the Sphere} \label{Sec: 4}

The construction in the last section above also suggests a family of integrable billiard systems on the sphere. 

We define billiard systems on the sphere with the spherical Kepler-Coulomb problem with either a circle whose centers agree with the centers of the Kepler-Coulomb problem (henceforth centered circles), or with any choice of a great circle as the wall of reflection, and with the natural law of reflection on the sphere, with dynamics on either side of the circle.  
When the antipodal pair of spherical Kepler-Coulomb singularities are contained in the circle of reflection, then we may either regularize collisions with the attractive center with elastic bouncing, or remove them from the circle. The repulsive center does not lie in any finite energy hypersurface of the spherical Kepler problem. 

\begin{prop} The spherical billiard systems thus defined are integrable.
\end{prop}
\begin{proof}

We first examined the case of centered circles.  We put the circle horizontal. In the case of a small circle in $S_{SH}$ we project it centrally to $V:=\{z=-1\}$. In this case we see just as in the planar case that the laws of reflections in the plane and on the sphere correspond to each other and thus both $E_{sph}$ and $E_{pl}$ are first integrals. In the case of a horizontal great circle, we may simply approximate the billiard system thus obtained by a family of billiard systems on the sphere with horizontal small circles. By a limiting argument,  both $E_{sph}$ and $E_{pl}$ are first integrals also in this case. Note that this pair of first integrals are equivalent to $E_{sph}$ together with the square of the angular momentum with respect to the axis passing through the centers, which is clearly also a conserved quantity.
 
When the great circle is not centered on the sphere, by rotation we may set the circle wall of reflection to lie in the plane $\{(x,0,z)\}$ and the centers to lie in the plane $\{(0, y, z)\}$. The intersection of the great circle wall of reflection with $S_{SH}$ is then centrally projected to a line in the planar system in $V$ and the relative position of the projected center in $V$ with respect to the south pole is perpendicular to the line in $V$. In this case the laws of reflection in the plane and on the sphere correspond to each other as has been observed in the proof of the planar case.
Therefore the planar and spherical energies are functionally independent conserved quantities of the corresponding billiard system on $S_{SH}$. Finally by analyticity the planar energy extends to a first integral of the billiard system on the sphere.
\end{proof}

\begin{figure}\label{Figure: Sph}
\center
\includegraphics[width=50mm]{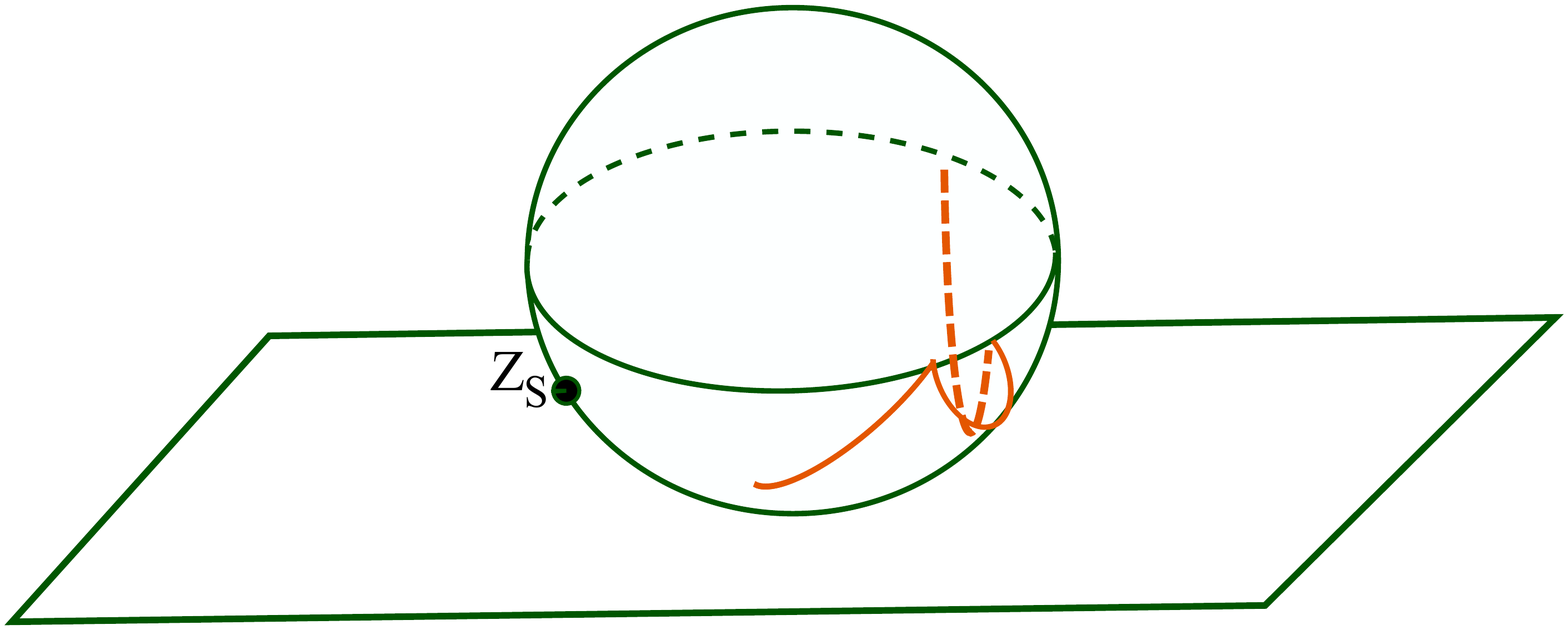}
\caption{A spherical billiard system}
\end{figure}



It is natural to consider these spherical systems on $S$ as extensions of the planar billiard systems proposed in Section \ref{Sec: Int Models plan}, in a similar way as the spherical Kepler-Coulomb problem extends the corresponding planar Kepler-Coulomb problem. In particular, the unbounded orbits in the integrable Boltzmann system is completed in its spherical corresponding system.

\section{Some further remarks}\label{Sec: 5}
We conclude with a few remarks.

First we remark that with similar construction we may get several families of integrable billiard systems with potentials on the pseudosphere as well. 

Next we remark that conformal transformations induce correspondences of mechanical systems \cite{Goursat}, \cite{Darboux}, which can be used to obtain other integrable billiard systems. As an example, with the conformal mapping $\C \to \C, z \mapsto z^{2}$, one gets from the integrable Boltzmann model an integrable billiard system in the plane with the harmonic potential of a pair of isotropic harmonic oscillators, with a hyperbola as wall of reflection. 

Moreover, the method of Darboux \cite{Darboux} can be used to obtain integrable billiard systems on a cone or on certain other surfaces of revolutions.

Finally, it can be interesting to better understand the integrable dynamics of all these systems. In particular, it can be  interesting to see to which extend the algebraic geometric method in \cite{Felder} can be applied to the systems on the sphere.

\medskip
\medskip
{\bf Acknowledgements}
Thanks to A. Albouy for helpful discussions and precise suggestion of references, and Gil Gor for his demand on clarification. The author is supported by DFG ZH 605/1-1.


\medskip
\medskip
\medskip
\medskip

Lei Zhao, University of Augsburg, Augsburg, Germany. Email: lei.zhao@math.uni-augsburg.de.


\begin{thebibliography}{99} 
\bibitem{Albouy} A. Albouy, Lectures on the Two-Body Problem, in \textsl{Classical and Celestial Mechanics: The Recife Lectures}, Princeton University Press, (2002).

\bibitem{Albouy 2008}
A. Albouy, Projective dynamics and classical gravitation,
\emph{Regul. Chaot. Dyn.}, 13:525-542, (2008).

\bibitem{Albouy 2013}
A. Albouy, There is a projective dynamics,
\emph{EMS Newsletter}, September:37-43 (2013).
 
 \bibitem {Appell} P. Appell, Sur les lois de forces centrales faisant d\'ecrire \`a leur point d'application une conique quelles que soient les conditions initiales, \textsl{Amer. J. Math.},13:153--158,  (1891).
 
 
\bibitem{Boltzmann} L. Boltzmann.  L\"osung  eines  mechanischen  Problems, \emph{Wiener Berichte},  58:1035-1044,(1868). \emph{Wissenschaftliche Abhandlungen}, Vol. I, p. 97-105.

\bibitem{BCRR}  R Brouzet, R Caboz, J Rabenivo and V Ravoson, Two degrees of freedom quasi-bi-Hamiltonian systems, \emph{ J. Phys. A: Math. Gen.} 29(9): 2069-2076, (1996).

\bibitem{Darboux} G. Darboux, Remarque sur la Communication pr\'ec\'edente. \emph{Comptes Rendus Acad. Sci. Paris}
\textbf{108}:449-450, (1889).

\bibitem{Felder}
G. Felder, Poncelet Property and Quasi-periodicity of the Integrable Boltzmann System, arXiv: 2008.03480, (2020).

\bibitem{Gallavotti 2014}
G. Gallavotti, \emph{Nonequilibrium and irreversibility}, Theoretical and Mathematical Physics. Springer-Verlag, (2014).

\bibitem{Gallavotti} G. Gallavotti, Ergodicity: a historical perspective. Equilibrium and Nonequilibrium, \emph{Eur. Phys. J. H.} \textbf{41}:181-259, (2016).

\bibitem{Gallavotti-Jauslin} G. Gallavotti, I. Jauslin,  A  Theorem  on  Ellipses,  an  Integrable  System  and  a Theorem of Boltzmann,  arXiv:2008.01955, (2020).

\bibitem{Goursat} E. Goursat, Les transformations isogonales en M\'ecanique. \emph{Comptes Rendus Acad. Sci. Paris}
\textbf{108}:446-448, (1889).

\bibitem{Halphen} G.H. Halphen, Sur les lois de Kepler, \emph{Bulletin de la Soci\'et\'e Philomatique de Paris}, \textbf{7}(1): 89-91, (1878).

\bibitem{levi-civita2} T. Levi-Civita, Sur la r\'egularisation du probl\`eme des trois corps. \emph{Acta Math.}, \textbf{42}:99-144, (1920).

\bibitem{P} P. Painlev\'e, M\'emoire sur la transformation des \'equations de la Dynamique, \emph{J. Math. Pures Appl.}, \textbf{4}(10): 5-92, (1894).

\bibitem{Serret} P. Serret,  \textsl{Th\'eorie Nouvelle G\'eom\'etrique et M\'ecanique des
Lignes \`a Double Courbure} Mallet-Bachelier, Paris, (1860).

\bibitem{Tabachnikov 1999} S. Tabachnikov, Projectively equivalent metrics, exact transverse linefields and the geodesic flow on the ellipsoid, \emph{Comment. Math. Helv.} \textbf{74}:306-321, (1999).

\bibitem{Tabachnikov 2002} S. Tabachnikov, Ellipsoids, complete integrability and hyperbolic geometry, \emph{Moscow Math. J.} 2:185-198, (2002).

\bibitem{Topalov-Matveev} V. S. Matveev, P. Topalov, Geodesic Equivalence and Integrability, \emph{MPIM Preprint Series} No. \textbf{74}, (1998). 

\end{thebibliography}
\end{document}